\documentclass[8pt,oneside,british]{amsart}
\usepackage[T1]{fontenc}
\usepackage[latin9]{inputenc}
\usepackage{color}
\usepackage{verbatim}
\usepackage{amstext}
\usepackage{amsthm}
\usepackage{amssymb}
\usepackage{esint}

\makeatletter
\numberwithin{equation}{section}
\numberwithin{figure}{section}
\theoremstyle{plain}
\newtheorem*{thm*}{\protect\theoremname}
\theoremstyle{plain}
\newtheorem{thm}{\protect\theoremname}
\theoremstyle{definition}
\newtheorem{defn}[thm]{\protect\definitionname}
\theoremstyle{remark}
\newtheorem*{rem*}{\protect\remarkname}
\theoremstyle{plain}
\newtheorem{fact}[thm]{\protect\factname}
\theoremstyle{plain}
\newtheorem{prop}[thm]{\protect\propositionname}
\theoremstyle{plain}
\newtheorem{cor}[thm]{\protect\corollaryname}
\theoremstyle{plain}
\newtheorem{lem}[thm]{\protect\lemmaname}
\theoremstyle{definition}
\newtheorem{example}[thm]{\protect\examplename}

\makeatother

\usepackage{babel}
\providecommand{\corollaryname}{Corollary}
\providecommand{\definitionname}{Definition}
\providecommand{\examplename}{Example}
\providecommand{\factname}{Fact}
\providecommand{\lemmaname}{Lemma}
\providecommand{\propositionname}{Proposition}
\providecommand{\remarkname}{Remark}
\providecommand{\theoremname}{Theorem}

\begin{document}
\title[Fluctuation of amenable ergodic averages]{Fluctuation bounds for ergodic averages of amenable groups}
\author{Andrew Warren}
\address{Carnegie Mellon University\\ 5000 Forbes Avenue\\ Pittsburgh, PA
15213\\ USA}
\email{awarren1@andrew.cmu.edu}
\begin{abstract}
We study fluctuations of ergodic averages generated by actions of
amenable groups. In the setting of an abstract ergodic theorem for
locally compact second countable amenable groups acting on uniformly
convex Banach spaces, we deduce a highly uniform bound on the number
of fluctuations of the ergodic average for a class of Følner sequences
satisfying an analogue of Lindenstrauss's temperedness condition.
Equivalently, we deduce a uniform bound on the number of fluctuations
over long distances for arbitrary Følner sequences. As a corollary,
these results imply associated bounds for a continuous action of an
amenable group on a $\sigma$-finite $L^{p}$ space with $p\in(1,\infty)$. 
\end{abstract}

\maketitle
In this article, we consider a problem at the interface of \emph{effective
ergodic theory }and the \emph{ergodic theory of group actions}. 

By \emph{effective ergodic theory}, we mean the following programme:
insofar as ergodic theory provides theorems which tell us about the
long-term behaviour of dynamical systems, we may ask for \emph{more
quantitative, or computationally explicit}, analogues of those theorems.
For instance, the mean ergodic theorem of von Neumann asserts that,
whenever $(S,\mu)$ is a $\sigma$-finite measure space, $T:S\rightarrow S$
is a measure-preserving transformation, and $f\in L^{2}(S,\mu)$,
the sequence of ergodic averages $A_{n}f:=\frac{1}{N}\sum_{i=1}^{N-1}f\circ T^{-i}$
converges in $L^{2}$ norm; but \emph{how fast} does this convergence
occur? 

Effective ergodic theory is complicated by some rather general negative
results. Indeed, it is known \cite{krengel1978speed} that in the
aforementioned setting of the mean ergodic theorem, there is \emph{no}
uniform rate of convergence of the averages $A_{n}f$ in norm if one
considers arbitrary functions $f$ in $L^{2}(S,\mu)$. Consequently,
if we are to find a quantitative analogue of the mean ergodic theorem,
viz. one which is more explicit about the nature of the convergence
of $A_{n}f$, it is necessary to look for more subtle convergence
data than a uniform rate of convergence. One such form of convergence
data is the \emph{number of $\varepsilon$-fluctuations} of a sequence
for each $\varepsilon>0$, or possibly $\varepsilon$-fluctuations
satisfying some side condition, such as the fluctuations being ``over
long distances''. This is exactly the type of convergence data that
we investigate here; a precise explanation of these terms is given
in Definition \ref{def:jumps.}.

On the other hand, the \emph{ergodic theory of group actions} seeks
to modify the machinery of classical ergodic theory, by replacing
the action of a single measure-preserving transformation with a measure-preserving
action by a group. In fact, since the action by a single invertible
measure-preserving transformation can be identified with a measure-preserving
action of $\mathbb{Z}$, many theorems in the ergodic theory of group
actions contain results from classical ergodic theory as special cases.
In particular, this is typically the case where one is interested
in the ergodic theory of a class of group actions where the groups
involved belong to a class of groups containing $\mathbb{Z}$, such
as abelian groups, nilpotent groups, or, in the case of the present
article, amenable groups (see Definition \ref{def:amenable} below).

Consider, therefore, the following version of the mean ergodic theorem
for actions of amenable groups:
\begin{thm*}
Let $L^{p}(S,\mu)$ be such that either $S$ is $\sigma$-finite and
$1<p<\infty$ or $\mu(S)<\infty$ and $p=1$, and let $x\in L^{p}(S,\mu)$.
Let $G$ be a locally compact second countable amenable group with
Haar measure $dg$, let $G$ act continuously on $(S,\mu)$ by measure-preserving
transformations, and let $(F_{n})$ be a Følner sequence of compact
subsets of $G$. Then $A_{n}x:=\frac{1}{|F_{n}|}\int_{F_{n}}\pi(g^{-1})xdg$
converges in $L^{p}$. 
\end{thm*}
This result is originally due to Greenleaf \cite{greenleaf1973ergodic},
whose proof goes by way of an abstract Banach space analogue of the
mean ergodic theorem which is simultaneously general enough to deduce
the mean ergodic theorem for an amenable group acting on any reflexive
Banach space or any $L^{1}(\mu)$ with $\mu$ a finite measure. Central
to Greenleaf's proof is a fixed point argument which, in particular,
does not give any effective convergence information about the averages
$A_{n}x$. (If one specialises to the case where $p=2$, a simpler
proof is available \cite[Thm. 8.13]{einsiedler2010ergodic}, but this
proof \emph{also} does not give any effective convergence information.)

Here our aim is to give an effective analogue of Greenleaf's theorem.
At the cost of some generality --- here, we only consider actions
of amenable groups on \emph{uniformly convex} Banach spaces --- we
obtain an explicit \emph{uniform fluctuation bound} for $(A_{n}x)$.
In other words, we deduce a result of the following form:
\begin{thm*}
\emph{(``Main theorem'') }Let $L^{p}(S,\mu)$ be such that $S$
is $\sigma$-finite and $1<p<\infty$, and let $x\in L^{p}(S,\mu)$.
Let $G$ be a locally compact second countable amenable group with
Haar measure $dg$, let $G$ act continuously on $(S,\mu)$ by measure-preserving
transformations, and let $(F_{n})$ be a Følner sequence of compact
subsets of $G$. Then $A_{n}x:=\frac{1}{|F_{n}|}\int_{F_{n}}\pi(g^{-1})xdg$
converges in $L^{p}$ \emph{with a uniform bound} on the number of
$\varepsilon$-fluctuations over long distances for each $\varepsilon>0$;
where the uniform bound (both the number of fluctuations and the ``long
distances'') depends exclusively, and explicitly, on: the choice
of $p\in(1,\infty)$, the quantity $\Vert x\Vert_{L^{p}(S,\mu)}$,
and the Følner sequence $(F_{n})$. 
\end{thm*}
In fact, the dependence of the number of $\varepsilon$-fluctuations
over long distances of $A_{n}x$ on the choice of Følner sequence
$(F_{n})$ is only via a specific type of data from the Følner sequence,
which expresses the \emph{fact} that $(F_{n})$ is a Følner sequence
in a slightly more quantitatively explicit fashion. We call this data
the\emph{ Følner convergence modulus}; see Definition \ref{def:convergence modulus}.
It also turns out to be possible to delete the ``over long distances''
clause from the main theorem if, instead, one adds a side condition
on the Følner sequence used, namely the condition that the Følner
sequence be ``fast'', as detailed in Definition \ref{def:fast}.

We deduce this theorem as a special case of an ``explicit fluctuation
bound'' analogue of an abstract mean ergodic theorem for a certain
class of amenable group actions on uniformly convex Banach spaces,
which we also call the ``main theorem'' of the article; this result
is stated as Theorem \ref{thm:abstract Main theorem} below, which
also gives the explicit form of the uniform bound. The modified version
of this theorem where the fluctuations are not ``over long distances''
but the Følner sequence is assumed to be ``fast'' is given as Corollary
\ref{cor:fast main theorem}.

The plan of the article is as follows. In Section 1, we establish
a number of background facts from functional analysis and the theory
of amenable groups which are needed to state and prove the main theorem.
In Section 2, we first provide a specialised proof of an abstract
mean ergodic theorem for lcsc amenable groups acting on uniformly
convex Banach spaces, and then modify this proof so as to be sufficiently
quantitatively explicit that we are able to deduce the main theorem.
Lastly, in Section 3, a discussion of some related literature and
open problems is provided.

\section{Preliminaries}
\begin{defn}
\label{def:jumps.}Fix an $\varepsilon>0$. Given a sequence $(x_{n})$
in some metric space, we say that $(x_{n})$ has at most $N$ $\varepsilon$-fluctuations
if for every finite sequence $n_{1}<n_{2}<\ldots<n_{k}$ such that
for each $1\leq i<k$, $d(x_{n_{i}},x_{n_{i+1}})\geq\varepsilon$,
it always holds that $k<N$. A weaker notion is ``$\varepsilon$-fluctuations
at distance $\beta$'': given some function $\beta:\mathbb{N}\rightarrow\mathbb{N}$
with $\beta(n)>n$ for every $n$, we say that $(x_{n})$ has at most
$N$ $\varepsilon$-fluctuations \emph{at distance }\textbf{\emph{$\beta$}}\textbf{
}if for every finite sequence $n_{1}<n_{2}<\ldots<n_{k}$ \emph{with
the property that $n_{i+1}\geq\beta(n_{i})$ for every $1\leq i<k$
}such that for each $1\leq i<k$, $d(x_{n_{i}},x_{n_{i+1}})\geq\varepsilon$,
it always holds that $k<N$. 
\end{defn}

\begin{rem*}
It holds that a sequence $(x_{n})$ is Cauchy (viz. that for every
$\varepsilon>0$ there exists an $N$ such that for $m,n\geq N$,
$d(x_{m},x_{n})<\varepsilon$) iff for every $\varepsilon>0$ there
exists some $N^{\prime}$ such that $(x_{n})$ has at most $N^{\prime}$
$\varepsilon$-fluctuations, iff for any $\varepsilon>0$ and $\beta$
with $\beta(n)>n$, there exists some $N^{\prime\prime}$ such that
$(x_{n})$ has at most $N^{\prime\prime}$ $\varepsilon$-fluctuations
at distance $\beta$. (More precisely: if $(x_{n})$ is Cauchy then
for \emph{any} $\beta:\mathbb{N}\rightarrow\mathbb{N}$ with $\beta(n)>n$,
$(x_{n})$ has only finitely many $\varepsilon$-fluctuations at distance
$\beta$ for each $\varepsilon>0$; whereas conversely, if there is
\emph{any} such $\beta$ so that $(x_{n})$ has only finitely many
$\varepsilon$-fluctuations at distance $\beta$ for each $\varepsilon>0$,
then it follows that $(x_{n})$ is Cauchy.) Likewise, it is obvious
that if for a specific sequence $(x_{n})$ we happen to know an explicit
$N$ witnessing the Cauchy property, then this $N$ also serves as
an explicit upper bound for $N^{\prime}$, and likewise any explicit
$N^{\prime}$ serves as an explicit upper bound on $N^{\prime\prime}$
(for any $\beta$). However the converses are all false in a strong
sense: there exist examples of sequences where $N^{\prime}$ is a
computable function of $\varepsilon$ but $N$ is not computable \cite{avigad2015oscillation},
and likewise with $N^{\prime\prime}$ (for $\beta\neq n+1$) and $N^{\prime}$
respectively\textcolor{red}{{} }\cite{kohlenbach2014fluctuations}. 

These phenomena are certainly present in ergodic theory. As mentioned
in the introduction, it has long been known that a single measure-preserving
transformation acting on $(S,\mu)$, then when $1\leq p<\infty$,
there exist functions $f\in L^{p}(S,\mu)$ for which the convergence
indicated by the mean (and pointwise) ergodic theorem occurs arbitrarily
slowly \cite{krengel1978speed}. In other words $(A_{n}x)$ does not
exhibit a uniform rate of convergence. However, it was shown by Avigad
and Rute \cite{avigad2015oscillation} that when $p\in(1,\infty)$
in this setting --- in fact, more generally, if the acted-upon space
$L^{p}(S,\mu)$ is replaced with any uniformly convex Banach space
$\mathcal{B}$ with modulus of uniform convexity $u(\varepsilon$)
(see definition below) --- then there exists a uniform bound on the
number of $\varepsilon$-fluctuations in the sequence $(A_{n}x)$
which depends only on $u(\varepsilon)$ and $\Vert x\Vert/\varepsilon$.
\end{rem*}
\begin{defn}
A normed vector space $(\mathcal{B},\Vert\cdot\Vert)$ is said to
be uniformly convex if there exists a nondecreasing function $u(\varepsilon)$
such that for all $x,y\in\mathcal{B}$ with $\Vert x\Vert\leq\Vert y\Vert\leq1$
and $\Vert x-y\Vert\geq\varepsilon$, it follows that $\left\Vert \frac{1}{2}(x+y)\right\Vert \leq\Vert y\Vert-u(\varepsilon)$.
Such a function $u(\varepsilon)$ is then referred to as a \emph{modulus
of uniform convexity }for $\mathcal{B}$. %
We say that $\mathcal{B}$ is \emph{$p$-uniformly convex} if $K\varepsilon^{p+1}$
is a modulus of uniform convexity for $\mathcal{B}$, where $K$ is
some constant. 
\end{defn}

\begin{rem*}
There are a number of equivalent ways to define uniform convexity.
We have chosen the preceding definition because it is the most convenient
for our argument, but it is worth mentioning another characterisation
which is perhaps more standard: a space $(\mathcal{B},\Vert\cdot\Vert)$
is uniformly convex provided there is a nondecreasing function $\delta(\varepsilon)$
(also called a modulus of uniform convexity) such that for all $x,y\in\mathcal{B}$
with $\Vert x\Vert,\Vert y\Vert\leq1$ and $\Vert x-y\Vert\geq\varepsilon$,
it follows that $\left\Vert \frac{1}{2}(x+y)\right\Vert \leq1-\delta(\varepsilon)$.
It is not hard to show (cf. \cite[Lemma 3.2]{kohlenbach2009quantitative})
that a function $\delta(\varepsilon)$ is a modulus of convexity in
this sense iff $u(\varepsilon)=\frac{\varepsilon}{2}\delta(\varepsilon)$
is a modulus of uniform convexity in the sense of our definition.
(This indicates the origin of the off-by-one issue in our definition
of $p$-uniform convexity.) 

It is well known \cite{clarkson1936uniformly} that the $L^{p}$ and
$\ell^{p}$ spaces are uniformly convex when $p\in(1,\infty)$. Hanner
showed \cite{hanner1956uniform} that for $p\in[2,\infty)$, the \emph{sharp}
modulus $\delta(\varepsilon)$ for $L^{p}$ has an especially nice
form, namely $\delta(\varepsilon)=1-\left(1-\left(\frac{\varepsilon}{2}\right)^{p}\right)^{1/p}$.
In particular, this implies that $u(\varepsilon)=\frac{\varepsilon}{2}-\frac{\varepsilon}{2}\left(1-\left(\frac{\varepsilon}{2}\right)^{p}\right)^{1/p}$
is a modulus of uniform convexity, in our sense, for $L^{p}$ with
$p\in[2,\infty)$. (The same work shows that, even though $L^{p}$
is also uniformly convex for $p\in(1,2)$, the sharp modulus $\delta(\varepsilon)$
does not have as nice of an explicit form.) 
\end{rem*}
We recall some basic notions from the theory of vector-valued integration.
We shall closely follow the recent textbook by Hytönen et al. \cite{hytonen2016analysis};
for the convenience of the reader, we will sometimes refer directly
to specific definitions, theorems, etc. therein. 

Consider some measure space $(G,\mathcal{A},\mu)$ with some function
$f:G\rightarrow\mathcal{B}$, with $(\mathcal{B},\Vert\cdot\Vert)$
a Banach space. We say that $f(g)$ is a \emph{simple function} with
respect to the $\sigma$-algebra $\mathcal{A}$ and space $\mathcal{B}$,
if it is of the form $\sum_{i=1}^{N}1_{A_{i}}(g)b_{i}$, with $1_{A_{i}}(g)$
an indicator function for $A_{i}\in\mathcal{A}$, and $b_{i}\in\mathcal{B}$.
We then say that a function $f$ is \emph{strongly measurable }if
it is a pointwise limit of simple functions, i.e. if there exists
a sequence $f_{n}$ of simple functions such that for every $g\in G$,
$\Vert f(g)-f_{n}(g)\Vert\rightarrow0$ \cite[Def. 1.1.4]{hytonen2016analysis}.
By contrast, a subtly different notion (but more standard in the vector
integration literature) is that of $\mu$-strong measurability, which
only asserts this limit for $\mu$-almost all $g\in G$, but requires
that the sets $A_{i}$ in the definition of simple function have finite
$\mu$-measure (see \cite[Def. 1.1.13 and Def. 1.1.14]{hytonen2016analysis}).
We do not directly use $\mu$-strong measurability in the main theorem
of this paper --- in particular, it is too weak of a form of measurability
for Propositions \ref{prop:precomposition} and \ref{prop:integrability}
below. The relationship between strong measurability and $\mu$-strong
measurability is the following (quoting from \cite[Prop. 1.1.16]{hytonen2016analysis}):
\begin{fact}
\label{fact:strongly measurable}Consider a measure space $(G,\mathcal{A},\mu)$,
a Banach space $\mathcal{B}$, and a function $f:G\rightarrow\mathcal{B}$.
\begin{enumerate}
\item If $f$ is strongly $\mu$-measurable, then $f$ is $\mu$-almost
everywhere equal to a strongly measurable function.
\item If $\mu$ is $\sigma$-finite and $f$ is $\mu$-almost everywhere
equal to a strongly measurable function, then $f$ is strongly $\mu$-measurable.
\end{enumerate}
\end{fact}

As a particular case of Fact \ref{fact:strongly measurable}, we see
that when $\mu$ is $\sigma$-finite, a strongly measurable function
is also $\mu$-strongly measurable.

For $\mu$-strongly measurable functions, one can define a form of
integration, namely the \emph{Bochner integral}, in direct analogy
with the Lebesgue integral. Bochner integration is denoted by $\int_{G}f(g)d\mu$.
A function is \emph{Bochner $\mu$-integrable }iff it is both $\mu$-strongly
measurable and $\int_{G}\Vert f(g)\Vert d\mu<\infty$, that is, $\Vert f\Vert:G\rightarrow\mathbb{R}$
is integrable in the Lebesgue sense \cite[Prop. 1.2.2]{hytonen2016analysis}.

We record some other basic facts about the Bochner integral. (Each
of these will be ultimately used in the proof of Lemma \ref{lem:Laverage-of-averages}.)
\begin{fact}
Let $(G,\mu)$ be a measure space and $f:G\rightarrow\mathcal{B}$
be $\mu$-strongly measurable. 

(1) $\Vert\int_{G}f(g)d\mu\Vert\leq\int_{G}\Vert f(g)\Vert d\mu$.
\cite[Prop. 1.2.2]{hytonen2016analysis}

(2) If $T\in\mathcal{L}(\mathcal{B},\mathcal{B})$, then $T\left(\int_{G}f(g)d\mu\right)=\int_{G}Tf(g)d\mu$.
\cite[Eqn. 1.2]{hytonen2016analysis}

(3) Fubini's theorem holds for the Bochner integral. \cite[Prop. 1.2.7]{hytonen2016analysis}
In particular, if $(H,\nu)$ is another measure space, and $\mu$
and $\nu$ are $\sigma$-finite, and $F:G\times H\rightarrow\mathcal{B}$
is Bochner integrable, then
\[
\int_{G\times H}Fd\mu\times d\nu=\int_{H}\left(\int_{G}Fd\mu\right)d\nu=\int_{G}\left(\int_{H}Fd\nu\right)d\mu.
\]
\end{fact}

A more general notion than strong measurability is \emph{weak measurability}:
we say that $f:G\rightarrow\mathcal{B}$ is \emph{weakly measurable}
if for every $b^{*}\in\mathcal{B}^{*}$, the function $b^{*}\circ f:G\rightarrow\mathbb{R}$
is measurable (in the ordinary sense as a function from $(G,\mathcal{A},\mu)$
to $\mathbb{R}$ with the Borel $\sigma$-algebra). The following
classical result indicates when weak measurability implies strong
measurability.
\begin{prop}
\emph{(Pettis measurability criterion \cite[Thm. 1.1.6]{hytonen2016analysis})
}Let $(G,\mu)$ be a measure space and $\mathcal{B}$ a Banach space.
For a function $f:G\rightarrow\mathcal{B}$ the following are equivalent:

\begin{enumerate}
\item $f$ is strongly measurable.
\item f is weakly measurable, and $f(G)$ is separable in $\mathcal{B}$. 
\end{enumerate}
\end{prop}

An easy consequence of the Pettis measurability criterion is the following.
\begin{prop}
\label{prop:cts}If the measure space $(G,\mu)$ is also a separable
topological space and every Borel set in $G$ is $\mu$-measurable,
and $f:G\rightarrow\mathcal{B}$ is continuous, then $f$ is strongly
measurable.
\end{prop}

\begin{proof}
Observe that for any $b^{*}\in\mathcal{B}^{*}$, $b^{*}\circ f$ is
a composition of continuous functions, and is therefore continuous.
Hence $f$ is weakly measurable. Moreover, it holds that the continuous
image of a separable space is separable.
\end{proof}
Before the next proposition, it will be convenient to introduce the
following terminology. 
\begin{defn}
\label{def:strongly Borel}If $(G,\mathcal{A},\mu)$ is a measure
space which is also a topological space, and $\mathcal{A}$ extends
the Borel $\sigma$-algebra on $G$, then we say that $f:G\rightarrow\mathcal{B}$
is \emph{weakly Borel} if $b^{*}\circ f:G\rightarrow\mathbb{R}$ is
Borel, for all $b^{*}\in\mathcal{B}^{*}$. Likewise we say that $f:G\rightarrow\mathcal{B}$
is \emph{strongly Borel} provided it is weakly Borel and that $f(G)$
is separable in $\mathcal{B}$. 
\end{defn}

Of course, in the case where $\mathcal{A}$ is precisely the Borel
$\sigma$-algebra on $G$, then the notions of weakly and strongly
Borel functions $f:G\rightarrow\mathcal{B}$ coincide with weak and
strong measurability of $f$. 
\begin{prop}
\label{prop:precomposition}If $(G,\mathcal{A},\mu)$ and $(H,\mathcal{C},\nu)$
are both topological spaces where every Borel set is measurable, $\phi:H\rightarrow G$
is Borel, and $f:G\rightarrow\mathcal{B}$ is strongly Borel, then
$f\circ\phi$ is strongly Borel (in particular strongly measurable).
\end{prop}

\begin{proof}
By hypothesis, $f$ is also weakly Borel, so for any $b^{*}\in\mathcal{B}^{*}$,
we have that $b^{*}\circ f$ is Borel. Therefore $b^{*}\circ(f\circ\phi)=(b^{*}\circ f)\circ\phi$
is Borel, and thus $f\circ\phi$ is weakly Borel. Now, let $A=\phi(H)$
be the image of $\phi$ in $G$. Note that $(f\circ\phi)(H)=f(A)$.
Since $f(G)$ is separable in $\mathcal{B}$, it follows that $f(A)$
is also separable in $\mathcal{B}$ since it is contained in $f(G)$. 
\end{proof}
We now fix some notation and terminology regarding group actions on
Banach spaces and measure spaces.

A locally compact group $G$ will always come equipped with a Haar
measure. In the countable discrete case this coincides with the counting
measure. Regardless of whether the group is discrete or continuous,
we will use the notations $dg$ and $\vert\cdot\vert$ interchangeably
to refer to the Haar measure. Conventions differ on the issue of whether
the measure space $(G,dg)$ is understood to come equipped with the
Borel $\sigma$-algebra or its completion; in the latter case, the
assertion that a function $f:G\rightarrow\mathcal{B}$ is strongly
Borel is stronger than the assertion that $f$ is strongly measurable.
Therefore, in what follows, we assume only that the $\sigma$-algebra
on $G$ \emph{contains} the Borel $\sigma$-algebra, and carefully
note when a function $f:G\rightarrow\mathcal{B}$ is assumed to be
strongly Borel rather than strongly measurable. We use the abbreviation
\emph{lcsc} for topological groups which are locally compact and second
countable; importantly, for lcsc groups it always holds that the Haar
measure is $\sigma$-finite.

In general, we say that a group $G$ acts on a Banach space $(\mathcal{B},\Vert\cdot\Vert)$
if there is a function $\pi(g)$ that returns an operator on $\mathcal{B}$
for every $g\in G$, $\pi(e)$ is the identity operator, and for all
$g,h\in G$, $\pi(g)\pi(h)=\pi(gh)$. Together these imply that $\pi(g)^{-1}=\pi(g^{-1})$.
We say that $G$ \emph{acts linearly on} $\mathcal{B}$ provided that
in addition, $\pi$ maps from $G$ to the space $\mathcal{L}(\mathcal{B},\mathcal{B})$
of linear operators on $\mathcal{B}$. %
Writing $\mathcal{L}_{1}(\mathcal{B},\mathcal{B})$ to indicate the
set of all linear operators from $\mathcal{B}$ to $\mathcal{B}$
with supremum norm $1$, another way to say that $G$ acts both linearly
and with unit norm on $\mathcal{B}$ is to say that $G$ acts on $\mathcal{B}$
via $\pi:G\rightarrow\mathcal{L}_{1}(\mathcal{B},\mathcal{B})$.\footnote{We remark that any group that acts via a representation $\pi:G\rightarrow\mathcal{L}(\mathcal{B},\mathcal{B})$
such that every $\pi(g)$ is \emph{nonexpansive} actually does so
via $\pi:G\rightarrow\mathcal{L}_{1}(\mathcal{B},\mathcal{B})$, by
the fact that $\pi(g^{-1})=\pi(g)^{-1}$ and the general fact about
linear operators that $\Vert T^{-1}\Vert\geq\Vert T\Vert^{-1}$. Nonexpansivity
is required for the proof of our main result.} Likewise, we say that a topological group $G$ \emph{acts continuously
on }$\mathcal{B}$ provided that for every $x\in\mathcal{B}$, if
$g\rightarrow e$ then $\Vert\pi(g)x-x\Vert\rightarrow0$. In other
words $g\mapsto\pi(g)x$ is continuous from $G$ to $\mathcal{B}$.
In the case where $G$ also acts linearly (resp. and with unit norm)
on $\mathcal{B}$, this is equivalent to requiring that $\pi:G\rightarrow\mathcal{L}(\mathcal{B},\mathcal{B})$
(resp. $\pi:G\rightarrow\mathcal{L}_{1}(\mathcal{B},\mathcal{B})$)
is continuous when $\mathcal{L}(\mathcal{B},\mathcal{B})$ is equipped
with the strong operator topology. 

Relatedly, in the case where a group $G$ acts on a measure space
$(S,\mathcal{A},\mu)$, we say that $G$ \emph{acts by measure-preserving
transformations} if $\mu(g^{-1}\cdot A)=\mu(A)$ for every $g\in G$
and $A\in\mathcal{A}$. Likewise, we say that $G$\emph{ acts continuously
on} $(S,\mathcal{A},\mu)$ provided that for every $A\in\mathcal{A}$,
if $g\rightarrow e$ then $\mu(A\Delta gA)\rightarrow0$. This notion
of continuous group action is related to our previous notion of continuous
action on a Banach space by Proposition \ref{prop:Koopman} below
(which may be taken as a justification for the terminology).

Furthermore, we say that if $G$ is understood as a measure space,
then $G$ acts \emph{strongly} on a Banach space\emph{ $\mathcal{B}$}
provided that for every $x\in\mathcal{B}$, $g\mapsto\pi(g)x$ is
strongly measurable from $G$ to $\mathcal{B}$.\footnote{By analogy with the previous situation where $G$ acts continuously
on $\mathcal{B}$: in the case where $G$ also acts linearly (resp.
and with unit norm) on $\mathcal{B}$, this is equivalent to requiring
that $\pi:G\rightarrow\mathcal{L}(\mathcal{B},\mathcal{B})$ (resp.
$\pi:G\rightarrow\mathcal{L}_{1}(\mathcal{B},\mathcal{B})$) satisfies
a ``strongly measurable with respect to the strong operator topology''
like condition: for a precise statement, see \cite[Cor. 1.4.7]{hytonen2016analysis}.
However we do not use this alternative, strong operator topology based,
characterisation in our argument. } Likewise, we say that $G$ \emph{acts Borel strongly on $\mathcal{B}$}
provided that for every $x\in\mathcal{B}$, $g\mapsto\pi(g)x$ is
strongly Borel from $G$ to $\mathcal{B}$. Our argument will simultaneously
require that $G$ acts Borel strongly on $\mathcal{B}$, and linearly
and with unit norm. To refer to this last condition, we will say that
$G$ acts Borel strongly on $\mathcal{B}$ via the representation
$\pi:G\rightarrow\mathcal{L}_{1}(\mathcal{B},\mathcal{B})$. 

In the situation where $G$ acts Borel strongly on $\mathcal{B}$
via the representation $\pi:G\rightarrow\mathcal{L}_{1}(\mathcal{B},\mathcal{B})$,
we can immediately deduce the following facts, which will be used
in the proof of Lemma \ref{lem:Laverage-of-averages}.
\begin{prop}
\label{prop:integrability}Suppose that $G$ acts Borel strongly on
$\mathcal{B}$ via the representation $\pi:G\rightarrow\mathcal{L}_{1}(\mathcal{B},\mathcal{B})$.
Then,
\begin{enumerate}
\item for each $A\subset G$ with $|A|<\infty$, we have that $1_{A}\pi(g^{-1})x$
is Bochner integrable for each $x\in\mathcal{B}$, i.e. $1_{A}\pi(g^{-1})x$
is $\mu$-strongly measurable and $\int_{A}\Vert\pi(g^{-1})x\Vert dg<\infty$.
Moreover, 
\item for each $A,B\subset G$ with $|A|,|B|<\infty$ and for each $x\in\mathcal{B}$,
we have that $1_{A\times B}\pi((hg)^{-1})x$ is Bochner integrable.
\end{enumerate}
\end{prop}

\begin{proof}
(1) Since $\pi(\cdot)x$ is a strongly Borel function from $G$ to
$\mathcal{B}$, and $g\mapsto g^{-1}$ is a continuous and therefore
Borel function from $G$ to $G$, it follows from Proposition \ref{prop:precomposition}
that $g\mapsto\pi(g^{-1})x$ is strongly measurable (in fact strongly
Borel). Since the Haar measure on $G$ is $\sigma$-finite, this implies
that $g\mapsto\pi(g^{-1})x$ is $\mu$-strongly measurable, by Fact
\ref{fact:strongly measurable}. For the second condition of Bochner
integrability, it suffices to observe that since $\pi(g^{-1})\in\mathcal{L}_{1}(\mathcal{B},\mathcal{B})$
for every $g\in G$,
\[
\int_{A}\Vert\pi(g^{-1})x\Vert dg\leq\int_{A}\Vert x\Vert dg<\infty.
\]
(2) Since $G$ is a topological group, it holds that group multiplication
is continuous. Therefore $(g,h)\mapsto\pi((hg)^{-1})x$ is strongly
measurable (in fact strongly Borel) thanks to Proposition \ref{prop:precomposition},
since it is a composition of the continuous (and therefore Borel)
multiplication function and the strongly Borel function $\pi(\cdot^{-1})x$.
Consequently, Fact \ref{fact:strongly measurable} implies that $(g,h)\mapsto\pi((hg)^{-1})x$
is also $\mu$-strongly measurable. We also have Bochner integrability
because again, 
\[
\int_{A\times B}\Vert\pi((hg)^{-1})x\Vert dg\times dh\leq\int_{A\times B}\Vert x\Vert dg\times dh<\infty.\qedhere
\]
\end{proof}
Let us tie all this discussion of Bochner integration and group actions
on Banach spaces back to our original setting of interest. In the
``concrete version'' of Greenleaf's mean ergodic theorem, stated
in the introduction, $G$ acts continuously and by measure-preserving
transformations on a $\sigma$-finite measure space $(S,\mu)$, and
$f\in L^{p}$ with $p\in(1,\infty)$. In the next proposition and
corollary, we briefly check that in fact, this particular situation
is actually covered by our abstract ``group action on Banach space''
framework.
\begin{prop}
\label{prop:Koopman}Suppose that $G$ is a topological group which
acts continuously and by measure-preserving transformations on a $\sigma$-finite
measure space $(S,\mu)$. Denote this action by 
\[
G\times S\rightarrow S
\]
\[
(g,s)\mapsto g\cdot s.
\]
Fix $p\in[1,\infty)$. 
\begin{enumerate}
\item This action of $G$ on $S$ induces an action of $G$ on $L^{p}(S,\mu)$,
given by 
\[
G\times L^{p}(S,\mu)\rightarrow L^{p}(S,\mu)
\]
\[
(g,f)\mapsto f\circ g^{-1}
\]
where for a fixed $g\in G$, the $L^{p}(S,\mu)$ function $f\circ g^{-1}$
is defined in the following manner: $(f\circ g^{-1})(s):=f(g^{-1}\cdot s)$.
In other words, there is a representation $\pi$ of $G$ into the
space of operators on $L^{p}(S,\mu)$, with $\pi(g)f=f\circ g^{-1}$. 
\item Moreover, $G$ acts linearly and \emph{isometrically} on $L^{p}(S,\mu)$,
in the sense that $\Vert\pi(g)f\Vert_{p}=\Vert f\Vert_{p}$ for all
$f\in L^{p}(S,\mu)$; so in particular, $\pi:G\rightarrow\mathcal{L}_{1}(L^{p}(S,\mu),L^{p}(S,\mu))$.
\item In addition, $G$ acts continuously on $L^{p}(S,\mu)$, that is, for
each $f\in L^{p}(S,\mu)$, $g\mapsto\pi(g)f$ is a continuous function
from $G$ to $L^{p}(S,\mu)$.
\end{enumerate}
\end{prop}

\begin{rem*}
This proposition is basically a modification of the rather classical
\emph{Koopman operator} formalism; except here, instead of a single
measure-preserving transformation, we have a topological group of
measure-preserving transformations. In the former case, the argument
is quite standard; its generalisation to a group action is straightforward,
but we explicitly go through the proof here for completeness. (The
argument is nearly given in \cite[Ch. 8]{einsiedler2010ergodic},
for instance, albeit using a more restrictive definition of continuous
$G$-action on a measure space.) 

Additionally, we note that the statement of the proposition includes
the case $p=1$, which is excluded from our main theorem; nor does
the proposition require that $G$ be lcsc or amenable. 
\end{rem*}
\begin{proof}
(1) Here, we need only to check explicitly that $(g,f)\mapsto f\circ g^{-1}$
is indeed a group action of $G$ on $L^{p}(S,\mu)$. To wit, we must
show that $\pi(e)f=f$, and $\pi(gh)f=\pi(g)\pi(h)f$, for every $f\in L^{p}(S,\mu)$.

The first is obvious, since 
\[
\pi(e)f(s)=(f\circ e^{-1})(s)=f(e^{-1}\cdot s)
\]
and $e^{-1}\cdot s=e\cdot s=s$ for every $s\in S$, since $(g,s)\mapsto g\cdot s$
is a group action on $S$. Likewise, the second condition also follows
directly from the fact that $(g,s)\mapsto g\cdot s$ is a group action
on $S$:
\[
\pi(gh)f(s)=(f\circ(gh)^{-1})(s)=f((gh)^{-1}\cdot s)=f(h^{-1}\cdot(g^{-1}\cdot s))=\pi(g)f(h^{-1}\cdot s)=\pi(g)\pi(h)f(s).
\]

(2) To see that $\pi(g)$ is a linear operator on $L^{p}(S,\mu)$
for each $g\in G$, simply note that 
\[
\pi(g^{-1})\left(cf_{1}+f_{2}\right)(s):=\left(cf_{1}+f_{2}\right)(g^{-1}\cdot s)=cf_{1}(g^{-1}\cdot s)+f_{2}(g^{-1}\cdot s)=c\pi(g^{-1})f_{1}+\pi(g^{-1})f_{2}.
\]
The isometry property follows immediately from the change of variable
formula and the fact that the action of $G$ on $(S,\mu)$ is measure
preserving. Explicitly: fix $h\in G$, and compute that 
\begin{align*}
\Vert\pi(h)f\Vert_{p}^{p} & =\int_{S}|f(h^{-1}\cdot s)|^{p}d\mu(s)=\int_{S}|f(s)|^{p}d\mu(h\cdot s)=\int_{S}|f(s)|^{p}d\mu(s)=\Vert f\Vert_{p}^{p}.
\end{align*}

(3) We need to show that, for arbitrary $f\in L^{p}(S,\mu)$, if $g\rightarrow e$
then $\Vert\pi(g)f-f\Vert_{p}\rightarrow0$. We first prove the claim
for $f$ an indicator function; then, for simple functions; then pass
to general functions in $L^{p}(S,\mu)$ by a density argument.

Fix a measurable set $A\subseteq S$ with $\mu(A)<\infty$. By the
assumption that the action of $G$ on $(S,\mu)$ is continuous, it
holds that: for every $\varepsilon>0$, there exists a $U\ni e$ such
that for all $g\in U$, $\mu(gA\Delta A)<\varepsilon$. %

Let $\chi_{A}(s)$ denote the indicator function for $A$. Let $U$
be a %
neighbourhood of the origin such that for every $g\in U$, $\mu(gA\Delta A)<\varepsilon^{p}$.
Let $g\in U$. Observe that $\pi(g)\chi_{A}=\chi_{A}(g^{-1}(s))=\chi_{gA}(s)$.
Since $g\in U$, we have that $\mu(A\Delta gA)<\varepsilon^{p}$.
But observe that 
\[
\Vert\chi_{A}-\chi_{gA}\Vert_{p}^{p}=\int_{A\Delta gA}1d\mu<\varepsilon^{p}.
\]
Since our choice of $\varepsilon$ was arbitrary, we conclude that
$g\mapsto\pi(g)f$ is continuous when $f$ is an indicator function
of a set of finite measure.

Now let $f$ be a simple function of the form $\sum_{i=1}^{k}c_{i}\chi_{A_{i}}$.
Then, $\pi(g)f=\sum_{i=1}^{k}c_{i}\chi_{gA_{i}}$. In this case, it
suffices to pick a neighbourhood $U\ni e$ which is small enough that
for each $i$, we have that 
\[
\mu(A_{i}\Delta g_{i}A_{i})<\left(\frac{\varepsilon}{kc_{i}}\right)^{p}.
\]
Indeed, from the triangle inequality, and the previous result for
indicator functions, we observe that 
\[
\Vert f-\pi(g)f\Vert_{p}\leq\sum_{i=1}^{k}\left\Vert c_{i}\left(\chi_{A_{i}}-\chi_{gA_{i}}\right)\right\Vert _{p}<\sum_{i=1}^{k}c_{i}\left(\frac{\varepsilon}{kc_{i}}\right)=\varepsilon.
\]

Lastly, take $f\in L^{p}(S,\mu)$ to be arbitrary. Suppose that $f_{0}$
is a simple function such that $\Vert f-f_{0}\Vert_{p}<\varepsilon/3$.
By the previous case, we can pick a %
neighbourhood $U\ni e$ such that $\Vert f_{0}-\pi(g)f_{0}\Vert_{p}<\varepsilon/3$
for all $g\in U$. Now, observe that 
\[
\Vert f-\pi(g)f\Vert_{p}\leq\Vert f-f_{0}\Vert_{p}+\Vert f_{0}-\pi(g)f_{0}\Vert_{p}+\Vert\pi(g)f_{0}-\pi(g)f\Vert_{p}.
\]
We have already seen that the first two terms are each $<\varepsilon/3$.
For the last term, observe that 
\[
\Vert\pi(g)f_{0}-\pi(g)f\Vert_{p}=\Vert\pi(g)(f_{0}-f)\Vert_{p}=\Vert f_{0}-f\Vert_{p}
\]
by part (2) of the proposition. Consequently, $\Vert\pi(g)f_{0}-\pi(g)f\Vert_{p}<\varepsilon/3$
as well, so that $\Vert f-\pi(g)f\Vert_{p}<\varepsilon$ as desired. 
\end{proof}

\begin{cor}
\label{cor:consistency}Let $G$ be a lcsc group acting continuously
by measure-preserving transformations on a $\sigma$-finite measure
space $(S,\mu)$. Fix $p\in[1,\infty)$. Then the induced action of
$G$ on $L^{p}(S,\mu)$ is linear, strongly Borel, and has unit norm. 
\end{cor}

\begin{proof}
That the induced action of $G$ on $L^{p}(S,\mu)$ is linear with
unit norm is immediate from Proposition \ref{prop:Koopman} (2). From
Proposition \ref{prop:Koopman} (3), we know that $g\mapsto\pi(g)f$
is continuous from $G$ to $L^{p}(S,\mu)$ for each $f$; since $G$
is separable, Proposition \ref{prop:cts} indicates that $g\mapsto\pi(g)f$
is strongly measurable. Thus, we need only to check that $g\mapsto\pi(g)f$
is strongly Borel, as per Definition \ref{def:strongly Borel}.

The continuity of $g\mapsto\pi(g)f$ and the separability of $G$
imply that the image of $G$ under this mapping is separable in $L^{p}(S,\mu)$.
Likewise, since $g\mapsto\pi(g)f$ is continuous, post-composing the
mapping with an element of the dual space of $L^{p}(S,\mu)$ results
in a continuous map from $G$ to $L^{p}(S,\mu)$, hence automatically
Borel also. Therefore $g\mapsto\pi(g)f$ is both weakly Borel and
has separable image, for any choice of $f$; so we've shown that $G$
acts Borel strongly on $L^{p}(S,\mu)$. 
\end{proof}
We now turn our attention to amenable groups. The following serves
as our preferred characterisation of amenability. 
\begin{defn}
\label{def:amenable}(1) Let $G$ be a countable discrete group. A
sequence $(F_{n})$ of finite subsets of $G$ is said to be a Følner
sequence if for every $\varepsilon>0$ and finite $K\subset G$, there
exists an $N$ such that for all $n\geq N$ and for all $k\in K$,
$|F_{n}\Delta kF_{n}|<|F_{n}|\varepsilon$. 

(2) Let $G$ be a locally compact second countable (lcsc) group with
Haar measure $|\cdot|$. A sequence $(F_{n})$ of compact subsets
of $G$ is said to be a Følner sequence if for every $\varepsilon>0$
and compact $K\subset G$, there exists an $N$ such that for all
$n\geq N$, there exists a subset $K^{\prime}$ of $K$ with $|K\backslash K^{\prime}|<|K|\varepsilon$
such that for all $k\in K^{\prime}$, $|F_{n}\Delta kF_{n}|<|F_{n}|\varepsilon$. 

\end{defn}

\begin{rem*}
It has been observed, for instance, by Ornstein and Weiss \cite{ornstein1987entropy}
that (2) is one of several equivalent ``correct'' generalisations
of (1) to the lcsc setting. Note however, that we do not assume $(F_{n})$
is nested ($F_{i}\subset F_{i+1}$ for all $i\in\mathbb{N}$) or exhausts
$G$ ($\bigcup_{n\in\mathbb{N}}F_{n}=G$), nor do we assume, in the
lcsc case, that $G$ is unimodular. (Each of these is a common additional
technical assumption when working with amenable groups.) Conversely,
some authors use a version of (2) where the sets in $(F_{n})$ are
merely assumed to have finite volume, rather than compact; thanks
to the regularity of the Haar measure, our definition results in no
loss of generality. 
\end{rem*}
\begin{defn}
If $G$ is either a countable discrete or lcsc amenable group, and
has some distinguished Følner sequence $(F_{n})$, and acts Borel
strongly on $\mathcal{B}$ via a representation $\pi:G\rightarrow\mathcal{L}_{1}(\mathcal{B},\mathcal{B})$,
then we define the $n$th \emph{ergodic average }operator as follows:
$A_{n}x:=\fint_{F_{n}}\pi(g^{-1})xdg$. Here $\fint$ denotes the
normalised integral, that is, $\fint_{A}f(g)dg:=\frac{1}{|A|}\int_{A}f(g)dg$. 
\end{defn}

\begin{prop}
With the notation above, $\Vert A_{n}\Vert_{\mathcal{L}(\mathcal{B},\mathcal{B})}\leq1$.
\end{prop}

\begin{proof}
Evidently $A_{n}$ is a linear operator from $\mathcal{B}$ to $\mathcal{B}$.
From Proposition \ref{prop:integrability} we know that $1_{F_{n}}\pi(g^{-1})x$
is Bochner integrable, and in particular

\[
\Vert A_{n}x\Vert:=\left\Vert \fint_{F_{n}}\pi(g^{-1})xdg\right\Vert \leq\fint_{F_{n}}\Vert\pi(g^{-1})x\Vert dg\leq\fint_{F_{n}}\Vert x\Vert dg=\Vert x\Vert.\qedhere
\]
\end{proof}
A key piece of quantitative information for us will be how large $N$
has to be if $g$ is chosen to be an element of $(F_{n})$, in order
for $|F_{N}\Delta gF_{N}|/|F_{N}|$ to be small. This information
is encoded by the following modulus:
\begin{defn}
\label{def:convergence modulus}Let $G$ be an amenable group, either
countable discrete or lcsc, with Følner sequence $(F_{n})$. A \emph{Følner
convergence modulus} $\beta(n,\varepsilon)$ for $(F_{n})$ returns
an integer $N$ such that:

\begin{enumerate}
\item If $G$ is countable discrete, 
\[
(\forall m\geq N)(\forall g\in F_{n})\left[|F_{m}\Delta gF_{m}|<|F_{m}|\varepsilon\right].
\]
\item If $G$ is lcsc, 
\[
(\forall m\geq N)(\exists F_{n}^{\prime}\subset F_{n})(\forall g\in F_{n}^{\prime})\left[|F_{n}\backslash F_{n}^{\prime}|<|F_{n}|\varepsilon\wedge|F_{m}\Delta gF_{m}|<|F_{m}|\varepsilon\right].
\]
\end{enumerate}
\end{defn}

We remark that if $(F_{n})$ is an \emph{increasing} Følner sequence
(that is, $F_{n}\subset F_{m}$ for all $n\leq n$) then it follows
trivially that $\beta(n,\varepsilon)$ is a nondecreasing function
for any fixed $\varepsilon$. However, in what follows we do not always
assume that $(F_{n})$ is increasing. In some instances it is technically
convenient to assume that $\beta(n,\varepsilon)$ is nondecreasing;
in this case, we can upper bound $\beta(n,\varepsilon)$ using an
``envelope'' of the form $\tilde{\beta}(n,\varepsilon)=\max_{1\leq i\leq n}\beta(i,\varepsilon)$.
Hence, in any case we are free to assume that $\beta(n,\varepsilon)$
is nondecreasing in $n$ if necessary. 

It should be clear that if we have a ``sufficiently explicit'' amenable
group $G$ with a ``sufficiently explicit'' Følner sequence $(F_{n})$,
then we can explicitly write down a Følner convergence modulus $\beta(n,\varepsilon)$
for $(F_{n})$. What is meant by this? Simply, the following: if we
know ``explicitly'' that $(F_{n})$ is a Følner sequence, this amounts
to knowing ``explicitly'', for each $g\in G$, that $|F_{n}\Delta gF_{n}|/|F_{n}|\rightarrow0$
(resp. the analogous statement for when $G$ is lcsc), which in turn
amounts to being able to write down how exactly this convergence occurs
in an explicit way, i.e. that one can explicitly write down a rate
of convergence $N(g,\varepsilon)$ for $|F_{n}\Delta gF_{n}|/|F_{n}|$.
But a Følner convergence modulus $\beta(n,\varepsilon)$ is just a
function which dominates $N(g,\varepsilon)$ for every $g\in F_{n}$
--- in particular we can take $\beta(n,\varepsilon)=\max_{g\in F_{n}}N(g,\varepsilon)$. 

Likewise (but less heuristically), it is easy to see that we can select
a Følner sequence in such a way that $\beta(n,\varepsilon)$ can be
chosen to be a computable function (for an appropriate restriction
on the domain of the second variable). The following argument has
essentially already been observed by previous authors \cite{cavaleri2017computability,cavaleri2018folner,moriakov2018effective}
working with slightly different objects, but we include it for completeness.
(In the interest of simplicity, we restrict our attention to countable
discrete groups; one can say something similar for lcsc groups with
computable topology, but we do not pursue this here.)
\begin{prop}
Let $G$ be a countable discrete finitely generated amenable group
with the solvable word property. Fix $k\in\mathbb{N}$. Then $G$
has a Følner sequence $(F_{n})$ such that $\beta(n,k^{-1})=\max\{n+1,3k\}$
is a Følner convergence modulus for $(F_{n})$. Moreover $(F_{n})$
can be chosen in a computable fashion.
\end{prop}

\begin{proof}
Fix a computable enumeration of the elements of $G$, as well as a
computable enumeration of the finite subsets of $G$. The solvable
word property ensures that we can do this, and also that the cardinality
of $F\Delta gF$ can always be computed for any $g\in G$ and finite
set $F$. So, take $F_{1}$ to be an arbitrary finite subset of $G$
containing $g_{1}$, the first element of $G$. Given $F_{n-1}$,
take $\tilde{F}_{n}$ to be the least (with respect to the enumeration)
finite subset of $G$ containing $F_{n-1}$, such that for all $g\in F_{n-1}$,
$|\tilde{F}_{n}\Delta g\tilde{F}_{n}|<|\tilde{F}_{n}|/n$. Such an
$\tilde{F}_{n}$ exists since $G$ is amenable. Then, put $F_{n}=\tilde{F}_{n}\cup\{g_{n}\}$
where $g_{n}$ is the $n$th element of $G$. This is indeed a Følner
sequence: for a fixed $g$, we see that $|F_{n}\Delta gF_{n}|/|F_{n}|<3/n$
for all $n$ greater than the first $N$ such that $g\in F_{N}$,
because
\[
\frac{|F_{n}\Delta gF_{n}|}{|F_{n}|}\leq\frac{2+|\tilde{F}_{n}\Delta g\tilde{F}_{n}|}{|F_{n}|}\leq\frac{2}{|F_{n}|}+\frac{|\tilde{F}_{n}\Delta g\tilde{F}_{n}|}{|\tilde{F}_{n}|}<\frac{3}{n}.
\]
 Hence $|F_{n}\Delta gF_{n}|/|F_{n}|\rightarrow0$. Moreover, we see
that if $m\geq\max\{n+1,3k\}$, then 
\[
(\forall g\in F_{n})\qquad|F_{m}\Delta gF_{m}|<3|F_{m}|/m\leq|F_{m}|/k
\]
and so $\beta(n,k^{-1}):=\max\{n+1,3k\}$ is indeed a Følner convergence
modulus for $(F_{n})$.
\end{proof}
\begin{rem*}
The previous proposition is not sharp. It has been shown that there
are groups \emph{without} the solvable word property which nonetheless
have computable Følner sequences with computable convergence behaviour
\cite{cavaleri2018folner}. (The cited paper uses a different explicit
modulus of convergence for Følner sequences than the present paper,
although the argument carries over to our setting without modification.) 
\end{rem*}
Finally, it will be convenient to introduce the notion of a Følner
sequence where the gap between $n$ and $\beta(n,\varepsilon)$ is
controlled. 
\begin{defn}
\label{def:fast}Let $G$ be a countable discrete or lcsc amenable
group and $(F_{n})$ a Følner sequence. Let $\lambda\in\mathbb{N}$
and $\varepsilon>0$. We say that $(F_{n})$ is a \emph{$(\lambda,\varepsilon)$-fast}
Følner sequence if

\begin{enumerate}
\item For $G$ countable and discrete, it holds that for all $n\in\mathbb{N}$
that for all $m\geq n+\lambda$, for all $g\in F_{n}$, $|F_{m}\Delta gF_{m}|/|F_{m}|<\varepsilon$.
\item For $G$ lcsc, it holds that for all $n\in\mathbb{N}$ that for all
$m\geq n+\lambda$, there exists a set $F_{n}^{\prime}\subset F_{n}$
such that $|F_{n}\backslash F_{n}^{\prime}|<|F_{n}|\varepsilon$,
so that for all $g\in F_{n}^{\prime}$, $|F_{m}\Delta gF_{m}|/|F_{m}|<\varepsilon$.
\end{enumerate}
In other words, a Følner sequence $(F_{n})$ is $(\lambda,\varepsilon)$-fast
provided that $\beta(n,\varepsilon)=n+\lambda$ is a Følner convergence
modulus for $(F_{n})$.
\end{defn}

It is clear that any Følner sequence can be refined into a $(\lambda,\varepsilon)$-fast
Følner sequence. 
\begin{prop}
Given $\lambda\in\mathbb{N}$ and $\varepsilon>0$, any Følner sequence
can be refined into a $(\lambda,\varepsilon)$-fast Følner sequence. 
\end{prop}

\begin{proof}
It suffices to produce a $(1,\varepsilon)$-fast refinement. For simplicity,
we only state the argument for the case where $G$ is countable and
discrete. 

Suppose we have already selected the first $j$ Følner sets in our
refinement $F_{n_{1}},\ldots,F_{n_{j}}$. Then, take $F_{n_{j+1}}$
to be the next element of the sequence $(F_{n})$ after $n_{j}$ such
that, for all $g\in\bigcup_{i=1}^{j}F_{n_{i}}$, $|F_{n_{j+1}}\Delta gF_{n_{j+1}}|/|F_{n_{j+1}}|<\varepsilon$.
Such a term exists since $(F_{n})$ is a Følner sequence. 
\end{proof}
Less obvious is the relationship between a Følner sequence being fast
and the property of being \emph{tempered} which is used in Lindenstrauss's
pointwise ergodic theorem \cite{lindenstrauss2001pointwise}, although
they are somewhat similar in spirit. Nonetheless, we quickly observe
that the previous proposition indicates that any tempered Følner sequence
can be refined into a Følner sequence which is both tempered and $(\lambda,\varepsilon)$-fast,
simply by the fact that any subsequence of a tempered Følner sequence
is again tempered. 

\section{The Main Theorem}

Frequently in ergodic theory, one argues that if $K\gg N$, then $A_{K}A_{N}x\approx A_{K}x$.
The following lemma makes this precise in terms of the modulus $\beta$. 
\begin{lem}
\label{lem:Laverage-of-averages}Let $(\mathcal{B},\Vert\cdot\Vert)$
be a normed vector space. Let $G$ be a lcsc amenable group with Følner
sequence $(F_{n})$, and let $G$ act Borel strongly on $\mathcal{B}$
via the representation $\pi:G\rightarrow\mathcal{L}_{1}(\mathcal{B},\mathcal{B})$.
Fix $N\in\mathbb{N}$ and $\eta>0$. Let $\beta$ be the Følner convergence
modulus and suppose $K\geq\beta(N,\eta)$. Then for any $x\in\mathcal{B}$,
$\Vert A_{K}x-A_{K}A_{N}x\Vert<3\eta\Vert x\Vert$. (If $G$ is countable
discrete, the ``strongly Borel'' part is trivially satisfied, and
we have the sharper estimate $\Vert A_{K}x-A_{K}A_{N}x\Vert<\eta\Vert x\Vert$.)
\end{lem}

\begin{proof}
From the definition of Følner convergence modulus, we know that there
exists an $F_{N}^{\prime}\subset F_{N}$ such that $|F_{N}\backslash F_{N}^{\prime}|<|F_{N}|\eta$
and such that for all $h\in F_{N}^{\prime}$, $|F_{K}\Delta hF_{K}|<|F_{K}|\eta$.
Now perform the following computation (justification for each step
addressed below): 
\begin{align*}
\Vert A_{K}x-A_{K}A_{N}x\Vert & :=\left\Vert \fint_{F_{K}}\pi(g^{-1})xdg-\fint_{F_{K}}\pi(g^{-1})\left(\fint_{F_{N}}\pi(h^{-1})xdh\right)dg\right\Vert \\
 & \;=\left\Vert \fint_{F_{K}}\pi(g^{-1})xdg-\fint_{F_{K}}\left(\fint_{F_{N}}\pi(g^{-1})(\pi(h^{-1})x)dh\right)dg\right\Vert \\
 & \;=\left\Vert \fint_{F_{K}}\pi(g^{-1})xdg-\fint_{F_{K}}\left(\fint_{F_{N}}\pi((hg)^{-1})xdh\right)dg\right\Vert \\
 & \;=\left\Vert \fint_{F_{N}}\left(\fint_{F_{K}}\pi(g^{-1})xdg\right)dh-\fint_{F_{N}}\left(\fint_{F_{K}}\pi((hg)^{-1})xdg\right)dh\right\Vert 
\end{align*}
\begin{align*}
 & \leq\fint_{F_{N}}\left\Vert \fint_{F_{K}}\pi(g^{-1})xdg-\fint_{F_{K}}\pi((hg)^{-1})xdg\right\Vert dh\\
 & =\fint_{F_{N}}\left\Vert \fint_{F_{K}}\pi(g^{-1})xdg-\fint_{hF_{K}}\pi(g^{-1})xdg\right\Vert dh\\
 & \leq\fint_{F_{N}}\left(\frac{1}{|F_{K}|}\int_{F_{K}\Delta hF_{K}}\Vert\pi(g^{-1})x\Vert dg\right)dh\\
 & \leq\fint_{F_{N}}\left(\frac{1}{|F_{K}|}\int_{F_{K}\Delta hF_{K}}\Vert x\Vert dg\right)dh\\
 & =\fint_{F_{N}}\frac{1}{|F_{K}|}\left(|F_{K}\Delta hF_{K}|\Vert x\Vert\right)dh\\
 & <\frac{1}{|F_{N}|}\left[\int_{F_{N}^{\prime}}\eta\Vert x\Vert dh+\int_{F_{N}\backslash F_{N}^{\prime}}\left(\frac{1}{|F_{K}|}|F_{K}\Delta hF_{K}|\Vert x\Vert\right)dh\right]\\
 & \leq\eta\Vert x\Vert+\frac{1}{|F_{N}|}\int_{F_{N}\backslash F_{N}^{\prime}}\left(2\Vert x\Vert\right)dh\leq3\eta\Vert x\Vert.
\end{align*}
If $G$ is countable discrete, we instead assume that for all $h\in F_{N}$
(rather than $F_{N}^{\prime}$), $|F_{K}\Delta hF_{K}|<|F_{K}|\eta$.
Therefore, the penultimate line reduces to $\frac{1}{|F_{N}|}\int_{F_{N}}\eta\Vert x\Vert dh$,
and the last line reduces to $\eta\Vert x\Vert$. 

Finally let's discuss which properties of the Bochner integral we
had to use. Thanks to Proposition \ref{prop:integrability}, we have
that $g\mapsto1_{F_{K}}\pi(g^{-1})x$ is Bochner integrable. Given
that, for each $g$, $\pi(g)$ is a bounded linear operator, then
indeed it follows that $\pi(g^{-1})\left(\int\pi(h^{-1})xdh\right)=\int(\pi(g^{-1})\pi(h^{-1})xdh$.
From Fubini's theorem and the fact (also from Proposition \ref{prop:integrability})
that $(g,h)\mapsto1_{F_{N}\times F_{K}}\pi((hg)^{-1})x$ is Bochner
integrable, we see that $\int_{F_{K}}\int_{F_{N}}\pi((hg)^{-1})xdhdg=\int_{F_{N}}\int_{F_{K}}\pi((hg)^{-1})xdgdh$.
Lastly, we repeatedly invoked the fact that $\Vert\int_{A}f(g)dg\Vert\leq\int_{A}\Vert f(g)\Vert dg$.
It's worth noting that in the case where $G$ is countable discrete,
only the first fact (that $G$ acts by bounded linear operators with
unit norm) is needed as an assumption, as the latter two properties
hold trivially for finite averages.
\end{proof}
\begin{rem*}
It is possible to generalise this argument to the case where the action
of $G$ is ``power bounded'' in the sense that there is some uniform
constant $C$ such that for ($dg$-almost) all $g\in G$, $\Vert\pi(g)\Vert\leq C$.
However the argument for our main theorem necessitates setting $C=1$.
\end{rem*}
The following argument is a generalisation of a proof of Garrett Birkhoff
\cite{birkhoff1939} to the amenable setting. The statement of the
theorem is weaker than results which are already in Greenleaf's article
\cite{greenleaf1973ergodic}, but we include the argument for several
reasons. One is that it is very short; another is that we will ultimately
derive a bound on $\varepsilon$-fluctuations via a modification of
this proof; and finally, the proof indicates additional information
about the limiting behaviour of the norm of $A_{n}x$, namely that
$\lim_{n}\Vert A_{n}x\Vert=\inf_{n}\Vert A_{n}x\Vert$.
\begin{thm}
Let $G$ be a lcsc amenable group with compact Følner sequence $(F_{n})$,
and let $\mathcal{B}$ be a uniformly convex Banach space such that
$G$ acts Borel strongly on $\mathcal{B}$ via the representation
$\pi:G\rightarrow\mathcal{L}_{1}(\mathcal{B},\mathcal{B})$. Then
for every $x\in\mathcal{B}$, the sequence of averages $(A_{n}x)$
converges in norm $\Vert\cdot\Vert_{\mathcal{B}}$.
\end{thm}

\begin{proof}
Without loss of generality, we assume $\Vert x\Vert\leq1$ (otherwise,
simply replace $x$ with $x/\Vert x\Vert$). Define $L:=\inf_{n}\Vert A_{n}x\Vert$.
Fix an arbitrary $\varepsilon_{0}>0$, and let $N$ be some index
such that $\Vert A_{N}x\Vert<L+\varepsilon_{0}$. Let $u$ denote
the modulus of uniform convexity. Fix a Følner convergence modulus
$\beta(n,\varepsilon)$ for $(F_{n})$, and let $\eta>0$ be arbitrary.
Suppose $M\geq\beta(N,\eta/(3\Vert x\Vert))$ is an index such that
$\Vert A_{N}x-A_{M}x\Vert>\delta$. (If no such $\delta$ exists then
this means that after $\beta(N,\eta/(3\Vert x\Vert))$, the sequence
has converged to within $\delta$.) Then, from uniform convexity,
we know that
\[
\left\Vert \frac{1}{2}(A_{N}x+A_{M}x)\right\Vert \leq\max\{\Vert A_{N}x\Vert,\Vert A_{M}x\Vert\}-u(\delta).
\]
Additionally, it follows from Lemma \ref{lem:Laverage-of-averages}
that $\Vert A_{M}x-A_{M}A_{N}x\Vert<\eta$, and thus 
\[
\left\Vert \frac{1}{2}(A_{N}x+A_{M}x)\right\Vert <\max\{\Vert A_{N}x\Vert,\Vert A_{M}A_{N}x\Vert+\eta\}-u(\delta).
\]
But $\Vert A_{M}A_{N}x\Vert\leq\Vert A_{N}x\Vert$, so this implies
\[
\left\Vert \frac{1}{2}(A_{N}x+A_{M}x)\right\Vert <\Vert A_{N}x\Vert+\eta-u(\delta).
\]
In turn, we know that $\Vert A_{N}x\Vert<L+\varepsilon_{0}$, so 
\[
\left\Vert \frac{1}{2}(A_{N}x+A_{M}x)\right\Vert <L+\varepsilon_{0}+\eta-u(\delta).
\]
In fact, it follows that $\Vert\frac{1}{2}A_{K}(A_{N}x+A_{M}x)\Vert<L+\varepsilon_{0}+\eta-u(\delta)$
also, for any index $K$, since $\Vert A_{K}\Vert\leq1$. Now, choosing
$K\geq\max\{\beta(N,\eta/(3\Vert x\Vert)),\beta(M,\eta/(3\Vert x\Vert))\}$,
we have that both $\Vert A_{K}x-A_{K}A_{N}x\Vert<\eta$ and $\Vert A_{K}x-A_{K}A_{M}x\Vert<\eta$.
Thus, 
\begin{align*}
\Vert A_{K}x\Vert & =\left\Vert \frac{1}{2}(A_{K}x-A_{K}A_{N}x)+\frac{1}{2}(A_{K}x-A_{K}A_{M}x)+\frac{1}{2}(A_{K}A_{N}+A_{K}A_{M})\right\Vert \\
 & \leq\eta+\left\Vert \frac{1}{2}A_{K}(A_{N}x+A_{M}x)\right\Vert \\
 & <2\eta+L+\varepsilon_{0}-u(\delta).
\end{align*}
Since $\eta$ can be chosen to be arbitrarily small (this merely implies
that $K$ and $M$ are very large), we see that $\limsup_{K}\Vert A_{K}x\Vert\leq L+\varepsilon_{0}-u(\delta)$.
But since our choice of $\varepsilon_{0}$ was arbitrary, it follows
that in fact $\limsup_{K}\Vert A_{K}x\Vert\leq L$. 

Moreover this implies that $(A_{n}x)$ converges in norm. For if this
were not the case, then we could find some $\delta_{0}>0$ such that
$\Vert A_{n}x-A_{m}x\Vert>\delta_{0}$ infinitely often; in fact,
there must be some $\delta_{0}>0$ such that $\Vert A_{n}x-A_{m}x\Vert>\delta_{0}$
infinitely often, \emph{with the further restriction}\footnote{Note, here, that we are invoking the the fact that a sequence in a
metric space converges iff, for any $\beta:\mathbb{N}\rightarrow\mathbb{N}$
with $\beta(n)>n$, the sequence contains only finitely many $\delta$-fluctuations
at distance $\beta$, for each $\delta>0$; cf. discussion of this
equivalence in the remark immediately following Definition \ref{def:jumps.}.} that $m\geq\beta(n,\eta/(3\Vert x\Vert)$. Now, pick $\eta$ and
$\varepsilon_{0}$ small enough that $2\eta+\varepsilon_{0}<u(\delta_{0})$,
and pick both $n$ and $m$ such that: $m$ and $n$ are sufficiently
large that $\Vert A_{n}x\Vert,\Vert A_{m}x\Vert<L+\varepsilon_{0}$
(which we can do since $\limsup_{K}\Vert A_{K}x\Vert\leq L$); and,
moreover, pick $m$ and $n$ in such a way that $\Vert A_{n}x-A_{m}x\Vert>\delta_{0}$
and $m\geq\beta(n,\eta/(3\Vert x\Vert))$. The computation above shows
that for $K$ larger than $\beta(m,\eta/(3\Vert x\Vert)$ and $\beta(n,\eta/(3\Vert x\Vert))$
we have that $\Vert A_{K}x\Vert<2\eta+L+\varepsilon_{0}-u(\delta_{0})$;
but since we chose $2\eta+\varepsilon_{0}<u(\delta_{0})$, this implies
that $\Vert A_{K}x\Vert<L$, which contradicts the definition of $L$.
\end{proof}
We now proceed to deriving a quantitative analogue of this result.
We first do so ``at distance $\beta$'', and then recover a global
bound in the case where the Følner sequence is fast (in the sense
of Definition \ref{def:fast}). The only really non-explicit part
of the proof of the preceding theorem was the step where we used the
fact that an infimum of a real sequence exists. Therefore the principal
innovation in what follows is the avoidance of the direct invocation
of this fact. 

Fix a nondecreasing Følner convergence modulus $\beta$ for $(F_{n})$.
Initially let us consider the case where $\Vert x\Vert\leq1$. Suppose
that $\Vert A_{n_{0}}x-A_{n_{1}}x\Vert\geq\varepsilon$. Moreover,
we pick some $\eta<u(\varepsilon)/2$, and suppose that $n_{1}\geq\beta(n_{0},\eta/3\Vert x\Vert)$.
Then the computation from the previous proof shows that
\[
\left\Vert \frac{1}{2}(A_{n_{0}}x+A_{n_{1}}x)\right\Vert <\Vert A_{n_{0}}x\Vert+\eta-u(\varepsilon).
\]
More generally, if $\Vert A_{n_{i}}x-A_{n_{i+1}}x\Vert\geq\varepsilon$
with $n_{i+1}\geq\beta(n_{i},\eta/3\Vert x\Vert)$, it follows that
\[
\left\Vert \frac{1}{2}(A_{n_{i}}x+A_{n_{i+1}}x)\right\Vert <\Vert A_{n_{i}}x\Vert+\eta-u(\varepsilon).
\]
Now, choosing $k\geq\max\{\beta(n_{i+1},\eta/(3\Vert x\Vert)),\beta(n_{i},\eta/(3\Vert x\Vert))\}$,
we have that both $\Vert A_{k}x-A_{k}A_{n_{i}}x\Vert<\eta$ and $\Vert A_{k}x-A_{k}A_{n_{i+1}}x\Vert<\eta$.
Thus, 
\begin{align*}
\Vert A_{k}x\Vert & =\left\Vert \frac{1}{2}(A_{k}x-A_{k}A_{n_{i}}x)+\frac{1}{2}(A_{k}x-A_{k}A_{n_{i+1}}x)+\frac{1}{2}(A_{k}A_{n_{i}}x+A_{k}A_{n_{i+1}}x)\right\Vert \\
 & \leq\eta+\left\Vert \frac{1}{2}A_{k}(A_{n_{i}}x+A_{n_{i+1}}x)\right\Vert \\
 & <2\eta+\Vert A_{n_{i}}x\Vert-u(\varepsilon).
\end{align*}
Therefore let $n_{i+2}$ equal the least index greater than $\max\{\beta(n_{i+1},\eta/(3\Vert x\Vert)),$
$\beta(n_{i},\eta/(3\Vert x\Vert))\}$ (equivalently, greater than
$\beta(n_{i+1},\eta/(3\Vert x\Vert))$, since $\beta$ is nondecreasing
in $n$) such that $\Vert A_{n_{i+1}}x-A_{n_{i+2}}x\Vert\geq\varepsilon$.
The previous calculation shows that $\Vert A_{n_{i+2}}x\Vert<\Vert A_{n_{i}}x\Vert+2\eta-u(\varepsilon)$.
More generally, we have that 
\[
\Vert A_{n_{i}}x\Vert<\Vert A_{n_{0}}x\Vert-\frac{i}{2}\left(u(\varepsilon)-2\eta\right)\quad i\text{ even}
\]
\[
\Vert A_{n_{i}}x\Vert<\Vert A_{n_{1}}x\Vert-\frac{i-1}{2}(u(\varepsilon)-2\eta)\quad i\text{ odd}
\]
So simply from the fact that $\Vert A_{n_{i}}x\Vert\geq0$, these
expressions derive a contradiction on the least $i$ such that 
\[
\max\left\{ \Vert A_{n_{0}}x\Vert,\Vert A_{n_{1}}x\Vert\right\} <\frac{i-1}{2}(u(\varepsilon)-2\eta)
\]
since this would imply that $\Vert A_{n_{i}}(x)\Vert<0$. That is,
the contradiction implies that the $n_{i}$th epsilon fluctuation
could not have occurred. We have no \emph{a priori} information on
the values of $\Vert A_{n_{0}}x\Vert$ and $\Vert A_{n_{1}}x\Vert$,
except that both are at most $\Vert x\Vert$. Therefore, we have the
following uniform bound: 
\[
i\leq\left\lfloor \frac{\Vert x\Vert}{\frac{1}{2}u(\varepsilon)-\eta}+1\right\rfloor 
\]
where $i$ tracks the indices of the subsequence along which $\varepsilon$-fluctuations
occur. This is actually \emph{one more} than the number of $\varepsilon$-fluctuations,
so instead we have that the number of $\varepsilon$-fluctuations
is bounded by $\left\lfloor \frac{\Vert x\Vert}{\frac{1}{2}u(\varepsilon)-\eta}\right\rfloor $.

If we happen to have any lower bound on the infimum of $\Vert A_{n}x\Vert$,
we can sharpen the previous calculation. Instead of using the fact
that for all $n$, $\Vert A_{n}x\Vert\geq0$, we use the fact that
$\Vert A_{n}x\Vert\geq L$ for some $L$. To wit, if $i$ is large
enough that 
\[
\Vert x\Vert<\frac{i-1}{2}(u(\varepsilon)-2\eta)+L
\]
then this would imply that $\Vert A_{n_{i}}x\Vert<L$, a contradiction.
Therefore we have the bound 
\[
i\leq\left\lfloor \frac{\Vert x\Vert-L}{\frac{1}{2}u(\varepsilon)-\eta}+1\right\rfloor 
\]
and so the number of $\varepsilon$-fluctuations is bounded by $\left\lfloor \frac{\Vert x\Vert-L}{\frac{1}{2}u(\varepsilon)-\eta}\right\rfloor $.

In the case where $\Vert x\Vert>1$, a small modification must be
made. We can make the substitution $x^{\prime}=x/\Vert x\Vert$, so
that $\Vert A_{n_{i}}x^{\prime}-A_{n_{i+1}}x^{\prime}\Vert\geq\varepsilon/\Vert x\Vert$,
and conclude (provided that $\eta<\frac{1}{2}u(\varepsilon/\Vert x\Vert)$
that $(A_{n}x^{\prime})$ has at most $\left\lfloor \frac{\Vert x^{\prime}\Vert}{\frac{1}{2}u(\varepsilon/\Vert x\Vert)-\eta}\right\rfloor $
$\varepsilon/\Vert x\Vert$-fluctuations at distance $\beta(n,\eta/3\Vert x^{\prime}\Vert)$,
or rather at most $\left\lfloor \frac{1}{\frac{1}{2}u(\varepsilon/\Vert x\Vert)-\eta}\right\rfloor $
$\varepsilon/\Vert x\Vert$-fluctuations at distance $\beta(n,\eta/3)$
since $\Vert x^{\prime}\Vert=1$. But since $\Vert A_{n_{i}}x^{\prime}-A_{n_{i+1}}x^{\prime}\Vert\geq\varepsilon/\Vert x\Vert$
iff $\Vert A_{n_{i}}x-A_{n_{i+1}}x\Vert\geq\varepsilon$, this actually
tells us that $(A_{n}x)$ has at most $\left\lfloor \frac{1}{\frac{1}{2}u(\varepsilon/\Vert x\Vert)-\eta}\right\rfloor $
$\varepsilon$-fluctuations at distance $\beta(n,\eta/3)$. Likewise,
if we happen to have a lower bound $L$ on the infimum of $\Vert A_{n}x\Vert$,
since $\inf_{n}\Vert A_{n}x\Vert\geq L$ iff $\inf_{n}\Vert A_{n}x^{\prime}\Vert\geq L/\Vert x\Vert$
we have that $(A_{n}x)$ has at most $\left\lfloor \frac{1-L/\Vert x\Vert}{\frac{1}{2}u(\varepsilon/\Vert x\Vert)-\eta}\right\rfloor $
$\varepsilon$-fluctuations at distance $\beta(n,\eta/3)$.

To summarise, we have shown that:
\begin{thm}
\emph{\label{thm:abstract Main theorem}(``Main theorem'')} Let
$\mathcal{B}$ be a uniformly convex Banach space with modulus $u$.
Fix $\varepsilon>0$ and $x\in\mathcal{B}$. Pick some $\eta>0$ such
that $\eta<\frac{1}{2}u(\varepsilon)$ if $\Vert x\Vert\leq1$, or
$\eta<\frac{1}{2}u(\varepsilon/\Vert x\Vert)$ if $\Vert x\Vert>1$.
Then if $G$ is a lcsc amenable group that acts Borel strongly on
$\mathcal{B}$ via the representation $\pi:G\rightarrow\mathcal{L}_{1}(\mathcal{B},\mathcal{B})$,
with Følner sequence $(F_{n})$, the sequence $(A_{n}x)$ has at most
$\left\lfloor \frac{\Vert x\Vert}{\frac{1}{2}u(\varepsilon)-\eta}\right\rfloor $
$\varepsilon$-fluctuations at distance $\beta(n,\eta/3\Vert x\Vert)$
if $\Vert x\Vert\leq1$, or at most $\left\lfloor \frac{1}{\frac{1}{2}u(\varepsilon/\Vert x\Vert)-\eta}\right\rfloor $
$\varepsilon$-fluctuations at distance $\beta(n,\eta/3)$ if $\Vert x\Vert>1$.
If we know that $\inf\Vert A_{n}x\Vert\geq L$, then we can sharpen
the bound to $\left\lfloor \frac{\Vert x\Vert-L}{\frac{1}{2}u(\varepsilon)-\eta}\right\rfloor $
in the case where $\Vert x\Vert\leq1$, or $\left\lfloor \frac{1-L/\Vert x\Vert}{\frac{1}{2}u(\varepsilon/\Vert x\Vert)-\eta}\right\rfloor $
in the case where $\Vert x\Vert>1$.
\end{thm}

\begin{cor}
\label{cor:fast main theorem}In the above setting, suppose that $(F_{n})$
is $(\lambda,\eta/3\Vert x\Vert)$-fast if $\Vert x\Vert\leq1$, or
$(\lambda,\eta/3)$-fast if $\Vert x\Vert>1$ respectively. Then the
sequence $(A_{n}x)$ has at most $\lambda\cdot\left\lfloor \frac{\Vert x\Vert}{\frac{1}{2}u(\varepsilon)-\eta}\right\rfloor +\lambda$
$\varepsilon$-fluctuations (resp. $\lambda\cdot\left\lfloor \frac{1}{\frac{1}{2}u(\varepsilon/\Vert x\Vert)-\eta}\right\rfloor +\lambda$
$\varepsilon$-fluctuations).
\end{cor}

\begin{proof}
We know from the theorem that, if $\Vert x\Vert\leq1$, there are
at most $\left\lfloor \frac{\Vert x\Vert}{\frac{1}{2}u(\varepsilon)-\eta}\right\rfloor $
$\varepsilon$-fluctuations at distance $\lambda$. This leaves the
possibility that there are some $\varepsilon$-fluctuations in the
$\left\lfloor \frac{\Vert x\Vert}{\frac{1}{2}u(\varepsilon)-\eta}\right\rfloor $
many gaps of width $\lambda$, and also that there are some $\varepsilon$-fluctuations
in between the last possible index $n_{i}$ given by the previous
theorem, and the index $n_{i+1}$ at which contradiction is achieved.
This last interval at the end is at most $\lambda$ wide as well.
The reasoning for the $\Vert x\Vert>1$ case is identical.
\end{proof}
\begin{example}
We mention two special cases of interest in the above setting. For
brevity we restrict ourselves to the case where $\Vert x\Vert\leq1$
and $L=0$, but the reader can make the obvious substitutions.

(1) In the case where $\mathcal{B}$ is $p$-uniformly convex with
constant $K$, then in this setting $(A_{n}x)$ has at most $\lambda\cdot\left\lfloor \frac{\Vert x\Vert}{\frac{K}{2}\varepsilon^{p+1}-\eta}\right\rfloor +\lambda$
$\varepsilon$-fluctuations.

(2) In the case where $\mathcal{B}=L^{p}(\mu)$ for $p\in[2,\infty)$,
then in this setting $(A_{n}x)$ has at most $\lambda\cdot\left\lfloor \frac{\Vert x\Vert}{\frac{\varepsilon}{4}-\frac{\varepsilon}{4}\left(1-\left(\frac{\varepsilon}{2}\right)^{p}\right)^{1/p}-\eta}\right\rfloor +\lambda$
$\varepsilon$-fluctuations.
\end{example}

\section{Discussion}

The prospect of obtaining quantitative convergence information for
ergodic averages via a modification of Garrett Birkhoff's argument
\cite{birkhoff1939} has been previously explored, in the $\mathbb{Z}$-action
setting, by Kohlenbach and Leu\c{s}tean \cite{kohlenbach2009quantitative}
and subsequently by Avigad and Rute \cite{avigad2015oscillation}. 

Our proof was carried out in the setting where the acted upon space
was assumed to be uniformly convex, and indeed our bound on the number
of fluctuations explicitly depends on the modulus of uniform convexity.
Nonetheless, it is natural to ask whether an analogous result might
be obtained for a more general class of acted upon spaces.

However, it has already been observed, in the case where $G=\mathbb{Z}$,
that there exists a separable, reflexive, and strictly convex Banach
space $\mathcal{B}$ such that for every $N$ and $\varepsilon>0$,
there exists an $x\in\mathcal{B}$ such that $(A_{n}x)$ has at least
$N$ $\varepsilon$-fluctuations \cite{avigad2015oscillation}. This
counterexample applies equally to bounds on the rate of metastability
(for more on metastability and its relationship with fluctuation bounds,
see for instance \cite[Section 5]{avigad2015oscillation}).\textcolor{red}{{}
}However, this counterexample does not directly eliminate the possibility
of a fluctuation bound for $\mathcal{B}=L^{1}(X,\mu)$, so the question
of a ``quantitative $L^{1}$ mean ergodic theorem for amenable groups''
remains unresolved. 

What about the choice of acting group? Our assumptions on $G$ (amenable
and countable discrete or lcsc) where selected because this is the
most general class of groups which have Følner sequences. Our argument
depends essentially on Følner sequences; indeed, proofs of ergodic
theorems for actions of non-amenable groups have a qualitatively different
structure. Remarkably, there are certain classes of non-amenable groups
whose associated ergodic theorems have much stronger convergence behaviour
than the classical ($G=\mathbb{Z}$) setting; for recent progress
on quantitative ergodic theorems in the non-amenable setting, we refer
the reader to the book and survey article of Gorodnik and Nevo \cite{gorodnik2009ergodic,gorodnik2015quantitative}.

We should also mention quantitative bounds for \emph{pointwise} ergodic
theorems. For $G=\mathbb{Z}$ such results go as far back as Bishop's
upcrossing inequality \cite{bishop1968constructive}. Inequalities
of this type have also been found for $\mathbb{Z}^{d}$ by Kalikow
and Weiss (for $F_{n}=[-n,n]^{d}$) \cite{kalikow1999fluctuations};
more recently Moriakov has modified the Kalikow and Weiss argument
to give an upcrossing inequality for symmetric ball averages in groups
of polynomial growth \cite{moriakov2018fluctuations}, and (simultaneously
with the preparation of this article) Gabor \cite{gabor2019fluctuations}
has extended this strategy to give an upcrossing inequality for countable
discrete amenable groups with Følner sequences satisfying a strengthening
of Lindenstrauss's temperedness condition.

For both norm and pointwise convergence of ergodic averages, it is
sometimes possible to deduce convergence behaviour which is stronger
than $\varepsilon$-fluctuations and/or upcrossings but weaker than
an explicit rate of convergence, namely that there exists an appropriate
\emph{variational inequalit}y. Jones et al. have succeeded in proving
numerous variational inequalities, both for norm and pointwise convergence,
for a large class of Følner sequences in $\mathbb{Z}$ and $\mathbb{Z}^{d}$
\cite{jones1998oscillation,jones2003oscillation}. (In particular,
the result proved in the present article is known not to be sharp
when specialised to the case where $G=\mathbb{Z}^{d}$ and $\mathcal{B}=L^{p}$,
due to these existing results.) However, their methods, which rely
on a martingale comparison estimate and a Calderón-Zygmund decomposition
argument, exploit numerous incidental geometric properties of $\mathbb{Z}^{d}$
which do not hold for many other groups. It would be interesting to
determine which other groups enjoy similar variational inequalities.

\section*{Acknowledgements}

The author thanks his advisor Jeremy Avigad, under whom this work
was completed, for his guidance and support. The author also thanks
Clinton Conley, Yves Cornulier, and Henry Towsner for helpful discussions.
In addition, the author thanks the anonymous referee, whose comments
led to a number of improvements to the exposition of the article.

\bibliographystyle{amsplain}
\bibliography{amenableergodic}

\providecommand{\bysame}{\leavevmode\hbox to3em{\hrulefill}\thinspace}
\providecommand{\MR}{\relax\ifhmode\unskip\space\fi MR }
\providecommand{\MRhref}[2]{%
  \href{http://www.ams.org/mathscinet-getitem?mr=#1}{#2}
}
\providecommand{\href}[2]{#2}
\begin{thebibliography}{10}

\bibitem{avigad2015oscillation}
Jeremy Avigad and Jason Rute, \emph{Oscillation and the mean ergodic theorem
  for uniformly convex {B}anach spaces}, Ergodic Theory and Dynamical Systems
  \textbf{35} (2015), no.~4, 1009--1027.

\bibitem{birkhoff1939}
Garrett Birkhoff, \emph{The mean ergodic theorem}, Duke Mathematical Journal
  \textbf{5} (1939), no.~1, 19,20.

\bibitem{bishop1968constructive}
Errett Bishop, \emph{A constructive ergodic theorem}, Journal of Mathematics
  and Mechanics \textbf{17} (1968), no.~7, 631--639.

\bibitem{cavaleri2017computability}
Matteo Cavaleri, \emph{Computability of {F}{\o}lner sets}, International
  Journal of Algebra and Computation \textbf{27} (2017), no.~07, 819--830.

\bibitem{cavaleri2018folner}
\bysame, \emph{F{\o}lner functions and the generic word problem for finitely
  generated amenable groups}, Journal of Algebra \textbf{511} (2018), 388--404.

\bibitem{clarkson1936uniformly}
James~A. Clarkson, \emph{Uniformly convex spaces}, Transactions of the American
  Mathematical Society \textbf{40} (1936), no.~3, 396--414.

\bibitem{einsiedler2010ergodic}
Manfred Einsiedler and Thomas Ward, \emph{Ergodic theory: with a view towards
  number theory}, vol. 259, Springer, 2010.

\bibitem{gabor2019fluctuations}
Uri Gabor, \emph{Fluctuations of ergodic averages for amenable group actions},
  arXiv preprint arXiv:1902.07912 (2019).

\bibitem{gorodnik2009ergodic}
Alexander Gorodnik and Amos Nevo, \emph{The ergodic theory of lattice
  subgroups}, Princeton University Press, 2009.

\bibitem{gorodnik2015quantitative}
\bysame, \emph{Quantitative ergodic theorems and their number-theoretic
  applications}, Bulletin of the American Mathematical Society \textbf{52}
  (2015), no.~1, 65--113.

\bibitem{greenleaf1973ergodic}
Frederick~P. Greenleaf, \emph{Ergodic theorems and the construction of summing
  sequences in amenable locally compact groups}, Communications on Pure and
  Applied Mathematics \textbf{26} (1973), no.~1, 29--46.

\bibitem{hanner1956uniform}
Olof Hanner, \emph{On the uniform convexity of ${L}^p$ and $l^p$}, Arkiv
  f{\"o}r Matematik \textbf{3} (1956), no.~3, 239--244.

\bibitem{hytonen2016analysis}
Tuomas Hyt{\"o}nen, Jan van Neerven, Mark Veraar, and Lutz Weis, \emph{Analysis
  in {B}anach spaces: Volume {I}: Martingales and {L}ittlewood-{P}aley theory},
  vol.~63, Springer, 2016.

\bibitem{jones1998oscillation}
Roger~L. Jones, Robert Kaufman, Joseph~M. Rosenblatt, and M{\'a}t{\'e} Wierdl,
  \emph{Oscillation in ergodic theory}, Ergodic Theory and Dynamical Systems
  \textbf{18} (1998), no.~4, 889--935.

\bibitem{jones2003oscillation}
Roger~L. Jones, Joseph~M. Rosenblatt, and M{\'a}t{\'e} Wierdl,
  \emph{Oscillation in ergodic theory: higher dimensional results}, Israel
  Journal of Mathematics \textbf{135} (2003), no.~1, 1--27.

\bibitem{kalikow1999fluctuations}
Steven Kalikow and Benjamin Weiss, \emph{Fluctuations of ergodic averages},
  Illinois Journal of Mathematics \textbf{43} (1999), no.~3, 480--488.

\bibitem{kohlenbach2009quantitative}
Ulrich Kohlenbach and Lauren{\c{t}}iu Leu{\c{s}}tean, \emph{A quantitative mean
  ergodic theorem for uniformly convex {B}anach spaces}, Ergodic Theory and
  Dynamical Systems \textbf{29} (2009), no.~6, 1907--1915.

\bibitem{kohlenbach2014fluctuations}
Ulrich Kohlenbach and Pavol Safarik, \emph{Fluctuations, effective learnability
  and metastability in analysis}, Annals of Pure and Applied Logic \textbf{165}
  (2014), no.~1, 266--304.

\bibitem{krengel1978speed}
Ulrich Krengel, \emph{On the speed of convergence in the ergodic theorem},
  Monatshefte f{\"u}r Mathematik \textbf{86} (1978), no.~1, 3--6.

\bibitem{lindenstrauss2001pointwise}
Elon Lindenstrauss, \emph{Pointwise theorems for amenable groups}, Inventiones
  Mathematicae \textbf{146} (2001), no.~2, 259--295.

\bibitem{moriakov2018fluctuations}
Nikita Moriakov, \emph{Fluctuations of ergodic averages for actions of groups
  of polynomial growth}, Studia Mathematica \textbf{240} (2018), no.~3,
  255--273.

\bibitem{moriakov2018effective}
Nikita Moriakov, \emph{On effective {B}irkhoff's ergodic theorem for computable
  actions of amenable groups}, Theory of Computing Systems \textbf{62} (2018),
  no.~5, 1269--1287.

\bibitem{ornstein1987entropy}
Donald~S. Ornstein and Benjamin Weiss, \emph{Entropy and isomorphism theorems
  for actions of amenable groups}, Journal d'Analyse Math{\'e}matique
  \textbf{48} (1987), no.~1, 1--141.

\end{thebibliography}

\end{document}